\documentclass[12pt]{article}
\usepackage{geometry} 
\usepackage{amssymb}
\usepackage{amsmath}
\usepackage{amsthm}

\newtheorem{theorem}{Theorem}[section] 

 \numberwithin{equation}{section}

\newcommand{\norm}[1]{\left\Vert#1\right\Vert}
\def\cD{{\mathcal D}}

 \title{ {Selfadjoint and $m$ sectorial extensions of  Sturm-Liouville operators} 
}
\author{B M Brown \\
Cardiff School of Computer Science \& Informatics, \\
Cardiff University, 
Cardiff CF24 3AA, 
UK\  and  W D Evans  \\    School of Mathematics, Cardiff University, \\ Senghennydd Road, Cardiff CF24 4AG  UK}

\date{} 

\begin{document}
\maketitle
%
%
%
%
%
%
%
%
%



\begin{abstract}
The self-adjoint and $m$-sectorial extensions of coercive Sturm-Liouville operators are characterised, under minimal smoothness conditions on the coefficients of the differential expression.
\end{abstract}


\section{Introduction}

A non-negative, densely defined symmetric operator $T$ acting in a Hilbert space $H$, has two distinguished self-adjoint extensions, namely the Friedrichs extension $T_F$ and the Krein extension $T_N$; for a strictly positive operator $T$, as will be the case in this paper, $T_N$ coincides with the von Neumann extension of $T$ and is accordingly called the Krein-von Neumann extension. These extensions are extremal in the sense that all non-negative self-adjoint extensions $\tilde{T}$ of $T$ satisfy $T_N \le \tilde{T} \le T_F$ in the form sense; recall that if $T_1,~T_2$ are non-negative self-adjoint operators with associated quadratic forms $t_1[\cdot], t_2[\cdot]$ and form domains $Q(T_1), Q(T_2)$ respectively, then $T_1 \le T_2$ in the form sense if $Q(T_2) \subset Q(T_1)$ and $ 0 \le t_1[u] \le t_2[u]$ for all $ u \in Q(T_2)$. The abstract theory was initiated by Krein in \cite{Kreina,Kreinb} and developed further by  Vishik in \cite{Vishik} and then by Birman in \cite{Birman}; we shall refer to the work of the 3 authors in the cited papers as the Krein-Vishik-Birman theory. In the first, and main, part of this paper we determine all the positive self-adjoint extensions when $T$ is the minimal operator generated by the general Sturm-Liouville expression 
\[
\tau  u := \frac{1}{k} \left\lbrace -(pu')' + q u \right\rbrace 
\]
on the interval $ [a,m), -\infty <a <m \le \infty$, with the coefficients subject to the following minimal conditions:
\begin{enumerate}
\item $k,p >0\  a.e.\  \textrm{on}\   [a,m); k, 1/p \in L^1_{loc}[a,m);$
	\item $q \in L^1_{loc}[a,m) $ real-valued.
\end{enumerate}
We are assuming that $a$ is a regular end-point of the interval, but $m$ is allowed to be singular, i.e., at least one of the following is possible:
\[
m= \infty;\ \int^m k dx = \infty;\ \int^m \frac{1}{p} dx = \infty;\ \int^m q dx = \infty.
\]
The case of $a$ a singular end-point can be treated similarly.

Let $H$ denote the weighted space $L^2(a,m;k)$ with inner-product and norm
\[
(f,g):= \int_a^m f(x) \overline{g(x)}k(x) dx,\ \ \|f\|:=(f,f)^{1/2}.
\]
The \textit{minimal} operator $T$ generated by $\tau$ in $H$ is the closure of $T'$ defined by
\begin{align} 
  T' u = {}&\tau u;\ \mathcal{D}(T')= \{u:  u \in \mathcal{D}(\tau), u(a)=pu'(a) =0,\  u= 0 \ \textrm{in}\ (X,m),\nonumber \\
 {}& \textrm{for some}\ X >a \}
\end{align}
where
\[
\mathcal{D}(\tau) := \{ u:u \in AC_{loc}[a,m)\cap H, pu' \in AC_{loc}[a,m), \tau u \in H \}.
\]
If $m$ is regular, then $T$ has domain
\begin{equation}\label{Eq2}
	 \mathcal{D}(T)\!= \!\{u:  u, pu' \in AC[a,m], \tau u \in H,  u(a)=pu'(a) = u(m) = pu'(m) =0 \}.
\end{equation}
The adjoint $T^*$ of $T$ is the \textit{maximal} operator defined by 
\begin{equation}\label{Eq3}
T^*u=\tau u;\ \mathcal{D}(T^*)=\{u: u \in \mathcal{D}(\tau) \cap H\}.
\end{equation}

In the main part of the paper, dealing with a symmetric $T$ and its self-adjoint extensions, we shall be assuming that $T$ is positive, i.e., $T \ge \mu  >0$, meaning
\[
(Tu,u) \ge \mu \|u\|^2,\ \ u \in \mathcal{D}(T),\ \ \mu >0,
\]
and our objective is to contribute to the classical problem of characterising all the positive self-adjoint extensions of $T$. Our results extend previous ones in that we assume only the minimal requirements on the coefficients of $\tau$ and that $T$ is positive. In order to achieve this degree of generality we use a theorem of Kalf in \cite{Kalf} in which he gives necessary and sufficient conditions for $T$ to be bounded below, and a general characterisation of the Friedrichs extension of Sturm-Liouville operators in terms of weighted Dirichlet integrals. Kalf's result on the semi-boundedness of $T$ is a consequence of a result of Rellich in \cite{Rel} equating the lower semi-boundedness to the disconjugacy of the equation $\tau u = \lambda u$ near $m$. Rellich proves that if 
\begin{equation} \label{Eq4}
\tau u = \lambda u
\end{equation}
is non-oscillatory near $m$ for some $\lambda \in \mathbb{R}$ then $T$ is bounded below, while conversely, if $T$ is bounded below with lower bound $ \alpha ,$ then (\ref{Eq4}) is disconjugate for every $ \lambda < \alpha $, i.e. any solution $ u \neq 0$ of (\ref{Eq4}) with $ \lambda \le \alpha$ has at most one zero .
Before quoting Kalf's result, we recall that if (\ref{Eq4}) is non-oscillatory at $m$, there exist linearly independent solutions $f,g$ with the following properties:
\begin{enumerate}
	\item $ f(x) >0,\ \ g(x) > 0$ in $ [s,m) $ for some $ s \in (a,m)$;
	\item $pf', \ pg' \in AC[a,m)$;
	\item $ \int_s^m\frac{1}{pf^2} dx = \infty;\ \ \int_s^m \frac{1}{pg^2} dx < \infty;$
	\item $\lim_{x \rightarrow m}\frac{f(x)}{g(x)} = 0.$
\end{enumerate}
The functions $f,g$ are called the \textit{principal} and \textit{non-principal} solutions of (\ref{Eq4}), respectively. Kalf's theorem implies the following when $a$ is regular; see Corollary 2 in \cite{Roseberger}:

\bigskip

\noindent{\bf Theorem (Kalf)} The operator $T$ is bounded below if and only if there exist $\mu \in \mathbb{R}$ and a function $h \in AC_{loc}[a,m)$ which is such that $ph' \in AC_{loc}[a,m),~h >0$ in $[s,m)$ for some $s \in (a,m),  \int_s^m \frac{1}{ph^2} dx = \infty$ and 
\[
q \ge \frac{(ph')'}{h} + \mu k\ \ \textrm{in}\ \ [s,m).
\]
The Friedrichs extension $T_F$ of $T$ has domain
\begin{equation} \label{Eq5}
\mathcal{D}(T_F)= \left\lbrace u: u \in \mathcal{D}(\tau), u(a)=0, \int_s^m ph^2 \left|\left( \frac{u}{h}\right)'\right|^2 dx < \infty \right\rbrace . 
\end{equation}
It follows that the sesquilinear form associated with $T_F$ is
\begin{equation}\label{Eq6}
t_F[v,w]= \int_a^m \left\lbrace ph  ^2 \left( \frac{v}{h}\right)' \overline{\left( \frac{w}{h}\right)'} + q_h v \overline{w}\right\rbrace dx,	
\end{equation}
where $ q_h = q - \frac{(ph')'}{h}$.

We shall assume that $ \mu >0$ so that $T$ is positive. Kalf also allows for $a$ to be singular. It is observed in \cite{Kalf} that the  {\it principal solution}  assumption on $h$ can be replaced by a {\it non-principal solution}  assumption, i.e.,$ \int_s^m \frac{1}{ph^2}dx = \infty$ can be replaced by $\int_s^m \frac{1}{ph^2}dx < \infty$, in which case 
\begin{equation} \label{Eq7}
\mathcal{D}(T_F)=\! \left\lbrace u:\! u \in \mathcal{D}(\tau), u(a)=0, \lim_{x\rightarrow m} \frac{|u(x)|}{h(x)} = 0 ,\ \! \int_s^m ph^2 \left|\left( \frac{u}{h}\right)'\right|^2 \! dx < \infty \right\rbrace . 
\end{equation}

 Our principal tool is the result from the Krein-Vishik-Birman theory, that 
 there is a one to one  correspondence between the set of all positive self-adjoint extensions of $T$ and the  set of 
 all pairs $\{N_B,B\}$, where $N_B$ is a subspace of  $\ker T^*$, and  $B$ is a
 positive self-adjoint operators    in $N_B$; the Krein-von Neumann extension corresponds to $B=0$ and the Friedrichs extension to $B=\infty$,(i.e., $B$ acts (trivially) in a $0$-dimensional vector space). The reader is referred to \cite{Grubb1970} for a comprehensive treatment, and in particular to section 2 of \cite{Grubb83} in which an account is given of earlier works of the author; the latter are acknowledged in the addendum \cite{Alonso/Simon2} to the survey article \cite{Alonso/Simon}.     If $\tilde{T}$ is a positive self-adjoint extension of $T$, then $\tilde{T} = T_B$ for some $B$, where $T_B$ is associated with a form $t_B[\cdot]$ which satisfies
\begin{equation}\label{Eq8}
t_B = t_F + b, \ \ Q(T_B) = Q(T_F) \dotplus Q(B)
\end{equation}
and $b[\cdot]$ is the form of $B$; we use the notation $Q(A)$ to denote the form domain of a positive operator $A$, i.e., the domain of the associated form $a[\cdot]$. Thus any $v \in Q(T_B)$ can be written $v=u + \eta$, where $u \in Q(T_F)$ and $\eta \in Q(B)\subseteqq N$. Furthermore
\begin{equation}\label{Eq9}
t_B[v] = t_F[u] + (B\eta, \eta);
\end{equation}
see \cite{Birman}, \cite{Grubb}, \cite{Grubb1970} and \cite{Alonso/Simon}. The Friedrichs and Krein-von Neumann extensions have  important roles in the application of the theory of symmetric operators. The Friedrichs extension has an acknowledged natural part to play in quantum mechanics, while in \cite{Grubb83}, Grubb describes an intimate connection between the eigenvalues of the Krein-von Neumann extension of a minimal elliptic differential operator of even order and those of a higher order problem concerning the buckling of a clamped plate. An abstract version of the latter connection is given in \cite{Gestazyetal}, establishing the Krein-von Neumann extension as a natural object in elasticity theory.

The last section of the paper is concerned with the more general problem of characterising all the coercive $m$-sectorial extensions of a coercive sectorial minimal operator $T$ generated by the expression $\tau$ with a complex coefficient $q$. The same problem is discussed by Arlinskii in \cite{arl_nach}, Theorem 3.1, where previously determined abstract results are applied to a general second-order differential operator on $[0,\infty)$ with $L^{\infty}[0,\infty)$ coefficients. Theorems 2.1 and 2.2 below are analogous results for $\tau$ under the minimal conditions assumed here. In \cite{Grubb}, Grubb developed methods of Birman and Vishik to characterise extensions of adjoint pairs of operators with bounded inverses and applied the results to elliptic partial differential operators.

An important tool in Kalf's paper, as well as in this paper, is the \textit{Jacobi factorisation identity}
\begin{equation}\label{Eq10}
-(pu')' + \frac{(ph')'}{h}u = -\frac{1}{h}\left[ ph^2\left( \frac{u}{h}\right) '\right] '.
\end{equation}
	
\medskip

\section{Self-adjoint extensions}
\subsection{The limit-point case at $m$}
Since $a$ is regular and $T \ge \mu >0$, the deficiency index of $(T-\lambda)$, and hence the dimension of \textrm{ker}($T^* - \lambda$), is constant for all $\lambda \in \mathbb{R} \setminus [\mu,\infty)$. In the limit-point case of $\tau u = \lambda u $ at $m$ it therefore follows, in particular, that $ N:= \textrm{ker}(T^*)$ is of dimension $1$, and so, any $\eta \in N$ can be written $\eta=c\psi$,  where $\psi\in L^2[a,m;k)$ is real, $c \in \mathbb{C} $ and  $\tau \psi=0$. Our main result in the limit-point case is

\medskip

\begin{theorem}
	\label{Thm1} Let $\tau$ be in the limit point case at $m$. Then there is a one-one correspondence between the positive self-adjoint extensions of $T$ and the one-parameter family of operators $\{T_l\}, 0\le l \le \infty$, where $T_l$ is the restriction of $T^*$ to the domain
	\begin{equation}
	\mathcal{D}(T_l):= \left\lbrace v: v\in \mathcal{D}(T^*), (pv')(a)= [ p(a)  \psi'(a) + l \norm{\psi}^2 ] v(a)\right\rbrace .\label{Eq2.1}
	\end{equation}
The operator $T_0$ is the Krein-von Neumann extension and $T_{\infty}$ is the Friedrichs extension.	
\end{theorem}
\begin{proof}
 If $ \psi(a) = 0$, then $ \psi$ is an eigenfunction at zero of the self-adjoint extension $T_0$ of $T$ determined by the Dirichlet boundary condition $ u(a) = 0$. Thus $N= \ker T_0$ and $ \mathcal{D}(T) \dotplus N \subseteq \mathcal{D}(T_0)$. Since $ \mathcal{D}(T) \dotplus N $ is the domain of $T_N$ (see \cite{Alonso/Simon}, Example 3.2.), it follows that $T_0 = T_N$.  However by  \cite[Theorem 5 ]{YSZ2016}, $T_0$ is the Friedrichs extension of $T$ and therefore has no null space. Consequently $\psi(a)=0$ is not possible.

Thus we may assume hereafter that $\psi(a) \neq 0$ and, without loss of generality, that
\begin{equation}\label{Eq2.2}
\psi(a) =1
\end{equation}
Consequently, from (\ref{Eq8}), any $v=u+\eta \in Q(T_B)$ is uniquely expressible as
$$ v(x)=v(x)-v(a)\psi(x) + v(a)\psi(x)$$
with
$$u(x)=v(x)-v(a)\psi(x) \in Q(T_F)$$
and
$$\eta(x)=v(a)\psi(x)\in Q(B).$$
Also, from (\ref{Eq9})
$$ t_B[v]=\int_a^m\{ph^2 | \left ( \dfrac {u}{h} \right ) '|^2 + q_h |u|^2 \} dx + l|v(a)|^2 \norm{\psi}^2$$
where $q_h=q-\dfrac{( ph')'}h$ and  $b[\eta]=(B\eta,\eta)=l\norm{\eta}^2, 0 \le l \le \infty.$
Moreover, for $\varphi=\theta + \xi= \varphi-\varphi(a)\psi+ \varphi(a) \psi \in Q(T_B)$,  the sesquilinear form identity for $t_B[v, \varphi] $ associated with (\ref{Eq9}) is

\begin{align*} {}& t_B[v,\varphi]= \\
 {}& \int_a^m    \left \{ ph^2 \left ( \dfrac{v-v(a)\psi} h \right )'    \left ( \dfrac{   \overline{ \varphi-\varphi(a)\psi  }} h  \right )'   
+    q_h  ( v -v(a)\psi    )  
\left ( \overline {\varphi-\varphi(a)\psi}  \right )\right \}dx   \\
 {}&+ l v(a)\overline \varphi(a)\norm{\psi}^2,   \\
 \end{align*}
 \begin{align}
= {}&\int_a^m \left \{ ph^2 \left ( \dfrac vh\right )' \left ( \dfrac{ \overline \varphi}h\right )' +  q_h v \overline \varphi \right \} dx\nonumber \\
{}& -\overline \varphi(a)   \int_a^m \left \{ ph^2 \left ( \dfrac vh\right )'\left ( \dfrac{ \overline \psi}h\right )' + q_h v \overline \psi \right \} dx\nonumber \\
{}& -v(a)  \! \int_a^m \left \{ ph^2 \left ( \dfrac \psi h\right )'\left ( \dfrac{ \overline{ \varphi-\varphi(a)\psi}}h\right )' + \!q_h( \psi )( \overline{ \varphi-\varphi(a)\psi}) \right \} dx\nonumber \\
{}&+ l v(a)\overline \varphi (a) \norm{\psi}^2\nonumber \\
=:I_1+I_2+I_3+I_4. \label{Eq2.3}
\end{align}
By (\ref{Eq10}),
\[
\tau \psi = -\frac{1}{h}\left[ ph^2 \left( \frac{\psi}{h}\right) '\right] ' +q_h \psi = 0.
\]
It follows on integration by parts, and since $ \theta (a) = \varphi(a) - \varphi(a) \psi(a)=0$, that

\begin{align*}
I_2={}&-   \left  [ph^2   \left ( \dfrac vh \right )  \left ( \dfrac{ \overline{\xi} }h\right )' \right  ]_a^m,\ \ \xi = \varphi(a) \psi,\\
I_3={}&-v(a)  \left  [ ph^2 \left ( \dfrac \psi h \right ) '  \left (  \dfrac{ \overline{ \varphi-\varphi(a)\psi}}h\right ) \right]_a^m\\
 = {}& -v(a)  \left  [ ph^2 \left ( \dfrac \psi h \right ) '  \left (  \dfrac{ \overline{ \varphi-\varphi(a)\psi}}h\right ) \right](m).
\end{align*}

For $v \in \cD(T_B)\subset \cD(T^*)$ we have by (\ref{Eq10})
\begin{align}
t_B[v,\varphi]={}&(T_B v,\varphi)=\int_a^m \left \{  \left ( -(pv')' + q v\right ) \overline \varphi \right \} dx  \nonumber \\
={}& \int_a^m \left \{ -\dfrac1h [ph^2 \left ( \dfrac v h \right ) ' ]'+q_h v \right \} \overline \varphi dx\nonumber \\
={}& - \left [ph^2 \left ( \dfrac v h \right )' \left ( \dfrac{ \overline \varphi}h\right )\right ]_a^m + I_1. \label{Eq2.4}
\end{align}
Thus 
\begin{align*}
-\left [ph^2 \left ( \dfrac v h \right )' \left ( \dfrac{ \overline \varphi}h\right )\right ]_a^m \\
 = -\overline {\varphi (a)}\left [ph^2 \left ( \dfrac v h \right ) \left ( \dfrac{ \overline \psi }h\right )'\right ]_a^m -{}&v(a)  \left  [ ph^2 \left ( \dfrac \psi h \right ) '  \left (  \dfrac{ \overline{ \varphi-\varphi(a)\psi}}h\right ) \right](m) \\
+{}& l v(a) \overline {\varphi (a)}\norm{\psi}^2
\end{align*}
and
\begin{align}
0 ={}& ph^2  \left [   \left (  \frac{v}{ h} \right ) ' \left ( \dfrac{ \overline \varphi}h \right ) -  \left ( \dfrac v h \right )  \left ( \dfrac{\overline \varphi (a)   \psi }h \right ) ' \right ] (m)   \nonumber \\
{}&- v(a)   \left [ p h^2 \left   (  \dfrac \psi h \right   )  ' \left ( \dfrac{ \varphi -\overline \varphi (a) \psi}h\right  )\right  ] (m) \nonumber \\
{}& - ph^2 \left [   \left ( \dfrac v h \right ) '     \left ( \dfrac{ \overline \varphi}h \right ) - \left ( \dfrac v h \right )  \left ( \dfrac{ \overline \varphi (a)  \psi}h \right   ) ' \right  ] (a) \nonumber \\
 {}&+ l v(a) \overline \varphi (a) \norm{\psi}^2.\label{Eq2.5}
\end{align}
The value at $m$ of the right-hand side of (\ref{Eq2.5}) is
\begin{align*}
&ph^2  \left[ \left\lbrace \left(\frac{v}{h}\right)'\left(\frac{\overline{\varphi}}{h}\right)- \left(\frac{v}{h}\right)\left(\frac{\overline{\varphi}}{h}\right)'\right\rbrace +  \left(\frac{v}{h}\right) \left(\frac{\overline{\theta}}{h}\right)' - \left(\frac{\eta}{h}\right)'\left(\frac{\overline{\theta}}{h}\right)\right] \\
&=ph^2  \left[ \left\lbrace \left(\frac{v}{h}\right)'\left(\frac{\overline{\varphi}}{h}\right)- \left(\frac{v}{h}\right)\left(\frac{\overline{\varphi}}{h}\right)'\right\rbrace \right . \\
&+\left . \left\lbrace \left(\frac{\eta}{h}\right)\left(\frac{\overline{\theta}}{h}\right)'- \left(\frac{\eta}{h}\right)'\left(\frac{\overline{\theta}}{h}\right)\right\rbrace
+ \left(\frac{u}{h}\right)\left(\frac{\overline{\theta}}{h}\right)'\right] \\ 
&	= p [v'\overline \varphi-v \overline \varphi'](m) + p[\eta \overline \theta '-\eta' \overline \theta](m)+ ph^2
\left (\dfrac uh \right )\left (\dfrac{ \overline \theta}h \right )'. \end{align*}
Suppose now that $ \varphi \in \mathcal(T_B)$. Then the functions $v,u,\eta, \varphi,\theta, \xi$  are members of $\cD(T^*)$ and thus
$$ p [v'\overline \varphi-v \overline \varphi'](m) = p[\eta \overline \theta '-\eta' \overline \theta](m)=0 , $$ 
since we have assumed the limit-point condition at $m$; see \cite{EE}, Theorem III.10.13 or \cite{Naimark}, Section 18.3. Hence the value of the right-hand side of (\ref{Eq2.5}) at $m$ is
\[
 \left [ ph^2 \left (  \frac{\overline{\theta}}{h} \right ) ' \left (  \frac { u}{h} \right ) \right ] (m).
\]
Since $ u, \theta \in \mathcal{D}(T_F)$, we have
$$
(T_F\theta,u)=\int_a^m \left \{ ph^2 \left ( \dfrac \theta h \right ) ' \left ( \dfrac { \overline u} h \right ) ' + q_h \theta \overline u \right \} dx =: I_5.
$$
But as $u(a)= \theta(a) = 0$,
$$ (T_F \theta, u)= \!\int_a^m \left \{ -\dfrac1h \left  [ ph^2 \left ( \dfrac \theta h \right ) '\right ]' + q_h \theta  \right \} \overline u dx = \left [ ph^2 \left (  \dfrac \theta h \right ) ' \left (  \dfrac { \overline u}h \right ) \right ] \!(m)+I_5 .
$$
Therefore, it follows that
\begin{equation}\label{Eq2.6}
\left [ ph^2 \left (  \dfrac \theta h \right ) ' \left (  \dfrac { \overline u}h \right ) \right ] (m)=0.  
\end{equation}
Hence, we infer from (\ref{Eq2.5}) that if $v \in \cD(T_B)$, then  $\forall \varphi \in \cD(T_B)$
$$
\overline \varphi(a)\left \{  (pv')(a)-v(a)p(a)   \psi'(a)-l v(a) \norm{\psi}^2\right \}=0.
$$
If $T_B$ is not the Friedrichs extension of $T$ we have from \cite{YSZ2016}  that   there  $\exists \varphi \in \cD(T_B)$ such that  $\varphi(a)\neq 0$.
Hence    we have that any $v \in \mathcal{D}(T_B)$ satisfies the boundary condition
\begin{equation}\label{Eq2.7}
(pv')(a)= [ p(a) \psi'(a) + l \norm{\psi}^2 ] v(a)  .
\end{equation}
The real constants $l$ parameterise the operators $T_B$: $l=0$ corresponds to $B=0$ and thus the Krein-von Neumann extension of $T$, while $l= \infty$ corresponds to $B=\infty$ and hence the Friedrichs extension.
\end{proof}


\bigskip

\subsection {The case of $m$ regular or limit circle, and $\tau u=0$ non-oscillatory at $m$}
 Let $f,g$ be the principal and non-principal solutions respectively of $\tau u =0$.  From Remark 3 in \cite{Kalf},
 \begin{equation}\label{Eq3.1}
 \frac{u}{g}(m) := \lim_{x\rightarrow m} \frac{u(x)}{g(x)} =0,\ \ \forall u \in Q(T_F),
 \end{equation}
and from Corollary 1 in \cite{Roseberger}
\begin{equation}\label{Eq3.2}
	\frac{u}{f}(m) := \lim_{x\rightarrow m} \frac{u(x)}{f(x)}  \ \ \textrm{exists},\ \ \forall u \in Q(T_F).
\end{equation}
We still have that $u(a) =0$ for $u \in Q(T_F)$, and both
\begin{equation}\label{Eq3.3}
t_F[u] = \int_a^m \left\lbrace ph^2\left|\left( \frac{u}{g}\right)'\right|^2 + q_g|u|^2\right\rbrace  dx,
\end{equation}
where $ q_g=q-\dfrac{( pg')'}g >0$, and
\begin{equation}\label{Eq3.4}
	t_F[u] = \int_a^m \left\lbrace ph^2\left|\left( \frac{u}{f}\right)'\right|^2 + q_f|u|^2\right\rbrace  dx,
\end{equation}
where $ q_f=q-\dfrac{( pf')'}f >0$, are valid. 

 We now have $\textrm{dim} N =2$, and $\{f,g\}$ is a fundamental system of solutions of $\tau u=0$. The self-adjoint operators $B$ act in subspaces $N_B$ of $N$ which therefore may be of dimension $1$ or $2$. In the case $\textrm{dim} N_B =2$ of the next theorem,  $\{\psi_1,\psi_2\}$ is a real orthonormal basis of $N_B$.
 
 \medskip
 
 \begin{theorem}\label{Theorem 3.2} The positive self-adjoint extensions of $T$ which correspond to operators $B$ in the Krein-Vishik-Birman theory with $ \textrm{dim}N_B =1$ form a one-parameter family $T_{\beta}$ of restrictions of $T^*$ with domains
\begin{equation}\label{Eq3.4A}
\mathcal{D}(T_{\beta}):= \left\lbrace v \in \mathcal{D}(T^*): \left[ pg^2\left\lbrace \left(\frac{v}{g}\right) \left( \frac{\psi}{g}\right) '- \left(\frac{v}{g}\right)' \left( \frac{\psi}{g}\right) \right\rbrace  \right] _a^m \!=\! \beta v(a)  \|\psi \|^2 \right\rbrace ,
\end{equation} 	
where $\psi$ is a real basis of $N_B$ with $\psi(a) =1$. 	

 The self-adjoint extensions corresponding to operators $B$ with $\textrm{dim} N_B = 2$ form a family $T_{\beta}$, where now $ \beta$ is a matrix $(b_{j,k})_{j,k=1,2}$ of parameters, and the $T_{\beta}$ are restrictions of $T^*$ to domains
\begin{align}\label{Eq3.4B}
\mathcal{D}(T_{\beta}):={}&\left\lbrace v\in \mathcal{D}(T^*):   \left[ pg^2 \left\lbrace  \left(\frac{v}{g}\right) \left( \frac{\psi_k}{g}\right)' -  \left(\frac{v}{g}\right)' \left( \frac{\psi_k}{g}\right)\right\rbrace  \right] _a^m = \right .\nonumber \\
{}&\left . \sum_{j =1}^2b_{k,j}c_j,\ \  k=1,2 \right\rbrace ,
\end{align}
where $c_1$ and $c_2$ are determined by the values of $v$ at $a$ and $m$ by, 
\begin{eqnarray}\label{Eq3.4C}
\frac{v}{g} (a) &=&  \sum_{j=1}^2 c_j \frac{\psi_j}{g}(a) \nonumber \\
\frac{v}{g} (m) &=&  \sum_{j=1}^2 c_j \frac{\psi_j}{g}(m) .
\end{eqnarray}
 \end{theorem}
 \begin{proof}
 
 \noindent{\bf The case $\textrm{dim} N_B = 1$}
 
Let the real function ${ \psi}$ be a basis of $N_B$, and $B \eta = \beta~c \psi$ for $\beta \in \mathbb{R}_+$ and $ \eta = c \psi \in N_B,~ c \in \mathbb{C }$. Then, $ \psi = c_1 f + c_2 g$ for some $c_1, c_2 \in \mathbb{R}$, and $\lim_{x \rightarrow m} \frac{\psi(x)}{g(x)}= c_2$. Suppose that $\psi(a) =1$.

Let $v, \varphi \in Q(T_B)$; then $ v=u+\eta,~ \varphi = \theta + \xi $, where $u, \theta \in Q(T_F),~ \eta, \xi  \in Q(B)$ and as
\[
v=v-v(a)\psi + v(a) \psi,\ \ \varphi =\varphi - \varphi(a)\psi + \varphi(a)\psi,
\]
we have that
\[
u = v-v(a) \psi,\ \eta = v(a) \psi,\ \ \ \theta =\varphi - \varphi(a) \psi,\ \xi=   \varphi(a)\psi.
 \]
  We now have from (\ref{Eq8}) that
  \begin{equation}\label{Eq3.6}
t_B[v,\varphi]= \int_a^m \left \{ pg^2 \left ( \dfrac ug\right )' \left ( \dfrac{ \overline \theta}g\right )' + q_g u \overline \theta \right \} dx  + b(\eta,\xi).  
\end{equation}  
The argument following (\ref{Eq2.3}) can be repeated, with $g$ replacing $h$ and using the facts that now $ \left(\frac{u}{g}\right) (m) = \left(\frac{\theta }{g}\right) (m) = 0$, as well as $u(a) = \theta(a) =0$. The term corresponding to $I_3$ is now zero and the result is that
\begin{align}\label{Eq3.7A}
t_B[v,\varphi]={}&\int_a^m\left\lbrace pg^2\left( \frac{v}{g}\right) '\left( \frac{ \overline{\varphi}}{g}\right) ' +q_g v \overline{\varphi}\right\rbrace  dx \nonumber \\
-{}&  \overline{\varphi(a)}\left[ pg^2 \left( \frac{v}{g}\right) \left( \frac{\psi}{g}\right) '\right]_a^m + \beta v(a)\overline{\varphi(a)} \|\psi\|^2.
\end{align}

If $v \in \mathcal{D}(T_B)$, since $T_B \subset T^*$, 
\begin{align}\label{Eq3.7}
t_B[v,\varphi]={}&(T_B v,\varphi)=\int_a^m ( \tau v ) \overline \varphi  dx \nonumber \\
={}& \int_a^m \left \{ -\dfrac1g [pg^2 \left ( \dfrac v g \right ) ' ]'+q_g v \right \} \overline \varphi dx \nonumber\\
={}& - \left [pg^2 \left ( \dfrac v g \right )' \left ( \dfrac{ \overline \varphi}g\right )\right ]_a^m + \int_a^m \left \{ pg^2 \left ( \dfrac vg\right )' \left ( \dfrac{ \overline \varphi}g\right )' + q_g v \overline \varphi \right \} dx  \nonumber\\
={}& - \left [pg^2 \left ( \dfrac v g \right )' \left ( \dfrac{ \overline \xi}g\right )\right ]_a^m + \int_a^m \left \{ pg^2 \left ( \dfrac vg\right )' \left ( \dfrac{ \overline \varphi}g\right )' + q_g v \overline \varphi \right \} dx, 
\end{align}
since $ \frac{\theta}{g} (a) = \frac{\theta}{g}(m) =0$. We conclude from (\ref{Eq3.7A}) and (\ref{Eq3.7}) that
\[
\left[ pg^2\left\lbrace \left(\frac{v}{g}\right) \left( \overline{\varphi}(a)\frac{\psi}{g}\right) '- \left(\frac{v}{g}\right)' \left( \overline{\varphi}(a)\frac{\psi}{g}\right) \right\rbrace  \right] _a^m = \beta v(a) \overline{\varphi}(a) \|\psi \|^2.
\]
Since $ \varphi(a) = d \psi(a) = d$ for arbitrary $ d \in \mathbb{C},$ it follows that
\begin{equation}\label{Eq3.8}
\left[ pg^2\left\lbrace \left(\frac{v}{g}\right) \left( \frac{\psi}{g}\right) '- \left(\frac{v}{g}\right)' \left( \frac{\psi}{g}\right) \right\rbrace  \right] _a^m = \beta v(a)  \|\psi \|^2.
\end{equation}


\noindent{\bf The case $\textrm{dim}(N_B) = 2$}
Let $ \{\psi_1, \psi_2 \} $ be a real orthonormal basis for $N$. Then
\[
B \psi_j  = \sum_{k=1}^2 b_{k,j} \psi_k,\ j=1,2, 
\]
where 
\[
 b_{j,k} = \overline{b_{k,j}}, \ \ \ \  (B \psi_j,\psi_k) = b_{k,j}.
\]
If $ \eta = \sum_{j=1}^2 c_j \psi_j,~ \xi = \sum_{j=1}^2 d_j \psi_j$, then 
\[
(B \eta, \xi) = \sum_{j,k =1}^2 b_{k,j}c_j \overline{d_k},\ \ \ \ (\eta, \xi) = \sum_{j=1}^2 c_j \overline{d_j}.
\]
Furthermore, for some $\mu_{jk} \in \mathbb{R},~j,k = 1,2,$
\begin{equation}\label{Eq3.10}
\psi_j  =  \mu_{j,1} f + \mu_{j,2} g,\ \ j=1,2,\end{equation}
so that 
\begin{equation}\label{Eq3.11}
\frac{\psi_j}{g}(m) = \mu_{j,2},\ \ j=1,2.
\end{equation}

From (\ref{Eq3.6}) and (\ref{Eq3.7}), we now have
\[
\left[ pg^2 \left\lbrace  \left(\frac{v}{g}\right) \left( \frac{\overline{\xi}}{g}\right)' -  \left(\frac{v}{g}\right)' \left( \frac{\overline{\xi}}{g}\right)\right\rbrace  \right] _a^m = (B \eta, \xi) = \sum_{j,k =1}^2b_{k,j}c_j \overline{d_k}
\]
and so
\begin{eqnarray}\label{Eq3.12}
&\sum_{k=1}^2 \overline{d_k}\left[ pg^2 \left\lbrace  \left(\frac{v}{g}\right) \left( \frac{\psi_k}{g}\right)' -  \left(\frac{v}{g}\right)' \left( \frac{\psi_k}{g}\right)\right\rbrace  \right] _a^m   = (B \eta, \xi)  \nonumber \\
& \!=\! \sum_{j,k =1}^2b_{k,j}c_j \overline{d_k}.	
\end{eqnarray}
  Since $d_1$ and $d_2$ are arbitrary, we have
\begin{equation}\label{Eq3.13}  
  \left[ pg^2 \left\lbrace  \left(\frac{v}{g}\right) \left( \frac{\psi_k}{g}\right)' -  \left(\frac{v}{g}\right)' \left( \frac{\psi_k}{g}\right)\right\rbrace  \right] _a^m =  \sum_{j =1}^2b_{k,j}c_j,\ \  k=1,2.
\end{equation}  

In (\ref{Eq3.13}), $c_1$ and $c_2$ are determined by the values of $v$ at $a$ and $m$, for $v(a) = \eta(a)$ and $ \left( \frac{v}{g}\right) (m) = \left( \frac{\eta}{g}\right) (m)$ by (\ref{Eq3.1}). To be specific,
\begin{eqnarray}\label{Eq3.14}
\frac{v}{g} (a) &=&  \sum_{j=1}^2 c_j \frac{\psi_j}{g}(a) \nonumber \\
\frac{v}{g} (m) &=&  \sum_{j=1}^2 c_j \frac{\psi_j}{g}(m) .
\end{eqnarray}

\end{proof}
\bigskip

\noindent{\bf Remark}
If in the case $\textrm{dim}N_B=1$ of the preceding theorem  $\psi = c_1f$, then $\frac{\psi}{g}(m) =0$ and the boundary condition becomes
\[
(p v')(a)  = \left( \beta \|\psi\|^2\right) v(a) , 
\]
with no contribution from $m$ as in the LP case. We could then repeat the above analysis with $g$ replaced by $f$ to get
\begin{equation}\label{Eq3.15}
\left[ pf^2\left\lbrace \left(\frac{v}{f}\right) \left( \frac{\overline{\varphi}}{f}\right)' - \left(\frac{v}{f}\right)' \left( \frac{\overline{\varphi}}{f}\right) \right\rbrace   \right] _a^m = \beta v(a) \overline{\varphi}(a) \|\psi \|^2,
\end{equation}
where $  \frac{\overline{\varphi}}{f}(m) =  \frac{\overline{\theta}}{f}(m) + c_1  $. The equation (\ref{Eq3.15}) has to be satisfied for all $ \theta \in Q(T_F)$.
\bigskip

\section{Coercive sectorial operators}
\subsection{$m$-sectorial extensions}

Let
\[
\tau := \frac{1}{k} \left\lbrace -(pu')' + q u \right\rbrace,\ \ q = q_1 +iq_2, 
\]
and 
\[
\tau^+ := \frac{1}{k} \left\lbrace -(pu')' + \overline{q} u \right\rbrace,\ \ \overline{q} = q_1 -iq_2, 
\]
on the interval $[a,m)$. Each of the expressions $\tau$ and $\tau^+$ has a minimal and maximal operator associated with it, and the notation has to indicate this; the minimal operators $T(\tau),~T(\tau^+)$  will be denoted by $T,~T^+$ respectively. We shall assume that $q_1$ satisfies the minimal conditions 1. and 2. of the introduction and that 
\begin{equation}\label{Eq4.1}
q_{1,h}:=	q_1- \frac{(ph')'}{h}  \ge \nu k,\ \ 	|q_2|  \le (\tan~\alpha)\left\lbrace 	q_1- \frac{(ph')'}{h}\right\rbrace  ,
\end{equation}
for some $ \nu > 0$ and $ \alpha \in (0,\pi/2)$.  The minimal operators $T,  T^+$ are then \emph{coercive} and  \emph{sectorial}, i.e., the numerical range of $T$,
\[
\Theta(T):= \left\lbrace (T u,u): u \in \mathcal{D}(T)\right\rbrace 
\]
lies in the sector
\[
\Theta(\alpha,\nu):= \left\lbrace z=x+iy \in \mathbb{C}: x \ge \nu>0,\ |y| \le \tan  \alpha ( x - \nu)\right\rbrace 
\]
and the same is true for $T^+$ and its numerical range. The maximal operators are the adjoints of the minimal operators, and $T,~T^+$ form an \emph{adjoint pair} in the sense that
\begin{equation}\label{Eq4.2}
T \subset (T^+)^* ;\ \ \ T^+ \subset T^* .
\end{equation}
Furthermore,  $T, T^+$ are  $J$-symmetric with respect to the conjugation $ J: u\mapsto \overline{u}$, i.e.,
\[
 JTJ \subset T^*,\ \  J(T^+)J \subset (T^+)^*.
\]
.

We denote by $T_F$ and $ T_N$ the Friedrichs and Krein-von Neumann extensions of $T$, respectively, and use a similar notation for $T^+$. We recall that the Friedrichs extension $T_F$ of a sectorial operator $T$ is an $m$-sectorial operator associated with the  closure of the sesquilinear form $t[u,v] = (Tu,v)$, and is coercive if $T$ is coercive. The form domain is $Q(T) = Q(T_F) = Q(T_F^R)$, where $T_F^R$ is the \emph{real} part of $T$, i.e. the  positive self-adjoint operator associated with the form $t^R[u,v]:= \Re t[u,v] := \frac{1}{2}\left( t[u,v] + t^*[u,v] \right)$ , where $t^*[u,v]:=  \overline{t[v,u]} $; see \cite{Kato} for details.

In Arlinskii's construction of the $m$-sectorial extensions of a coercive sectorial operator $T$, the Krein-von Neumann extension $T_N$ has an important role. We refer to \cite{Arl} for the definition of $T_N$ and a comprehensive treatment. From Theorem 3.6 in \cite{Arl},  
\begin{equation}\label{Eq4.3}
Q(T_N) = Q(T) \dotplus N,\ \ N = \ker T^*
\end{equation}
and
\begin{equation}\label{Eq4.4}
t_N[u,v] = t[Pu,Pv],\ \ \ \forall u, v \in Q(T_N) ,
\end{equation}
where $P$ is the projection of $Q(T_N)$ onto $Q(T)$ with respect to the decomposition (\ref{Eq4.3}).
Furthermore
\begin{equation}\label{Eq4.5}
\mathcal{D}(T_N) = \mathcal{D}(T) \dotplus N;\ \ T_N(f+v) = Tf,\ f\in \mathcal{D}(T), v \in N.
\end{equation}
The identities (\ref{Eq4.3}) - (\ref{Eq4.5}) also have exact analogues for $T^+$.

From Edmunds/Evans Theorem III.10.7, 
\begin{equation}\label{Eq4.6}
2 \le \textrm{def}~T + \textrm{def}~T^+ \le 4,
\end{equation}
and since $\textrm{def}~T$ and $\textrm{def}~T^+ $ are equal, being the dimensions of the kernels of $ T^*,~(T^+)^*$ respectively, we have that
\begin{equation}\label{Eq4.7}
1 \le 	\dim(\ker{T^*}) = \dim(\ker{(T^+)^*}) \le 2.
\end{equation}

\subsection{The case $\dim(\ker{T^*}) =1$}

\begin{theorem}
	\label{Theorem 4.2} Let $\dim(\ker{T^*}) =1$ and let $ \psi \in \ker~T^*$ be such that $ \psi(a) =1.$ Then the Krein-von Neumann extension of $T$ has domain
	\begin{equation}\label{Eq4.8}
	\mathcal{D}(T_N):=  \left\lbrace v: v-v(a) \psi \in \mathcal{D}(T_F),\ v'(a) - v(a) \psi'(a) = 0\right\rbrace, 
	\end{equation}
and for all $ v \in \mathcal{D}(T_N)$,
	\begin{equation}\label{Eq4.9}
	T_N v = \tau\left( v - v(a) \psi\right) .
	\end{equation}
	Moreover, the form domain of $T_N$ is
	\begin{equation}\label{Eq4.10}
	Q(T_N) = \left\lbrace v: u= v - v(a) \psi \in Q(T) \right\rbrace 
	\end{equation}
	and 
	\begin{equation}\label{Eq4.11}
	t_N[v] = t[u] = \int_a^m \left\lbrace ph^2\left|\left( \frac{u}{h}\right)'\right|^2 + q_h|u|^2\right\rbrace  dx,\ \ u=v-v(a) \psi.
	\end{equation}
\end{theorem}
\begin{proof}
Any $ v \in Q(T_N)$ can be written
\[
v=v- v(a) \psi + v(a) \psi.
\]
Hence, since $u(a)=0$ for all $u \in Q(T_F)$, we have from (\ref{Eq4.3}) the unique representation $v=u+\xi $, where $ u = v-v(a) \psi,~ \xi = v(a) \psi$. It follows from (\ref{Eq4.4}) that
\begin{equation}\label{Eq4.12}
t_N[v,\varphi] = t_F[u,\theta],
\end{equation}
for all $ \varphi \in Q(T_N)$, with $ P\varphi = \theta = \varphi- \varphi(a) \psi$. Thus
\begin{equation}\label{Eq4.13}
t_N[v] = t_F[u]  = \int_a^m \left\lbrace ph^2\left|\left( \frac{u}{h}\right)'\right|^2 + q_h|u|^2\right\rbrace  dx,\ \ u=v-v(a) \psi.
\end{equation}
From (\ref{Eq4.5}), and since $u(a) = (pu')(a) = 0$ for $u \in \mathcal{D}(T)$, we have that $ v \in \mathcal{D}(T_N)$ satisfies the boundary condition
\begin{equation}\label{Eq4.14}
v'(a) - v(a) \psi'(a) = 0,\ \ \forall  v = v-v(a) \psi + v(a) \psi \in \mathcal{D}(T_N).
\end{equation}
Also
\begin{equation}\label{Eq4.15}
T_N v = T (v-v(a)\psi) = \tau (v-v(a)\psi).
\end{equation}
\end{proof}

\bigskip

We shall now follow Arlinskii's analysis in \cite{arl_nach}, Section 3.1,to characterise all coercive $m$-sectorial extension of $T$. Arlinskii considers a more general second-order differential expression to generate his operator $T$, but with $L^{\infty}$ coefficients, and not the minimal conditions we impose on ours.

Let $X_0$ denote $Q(T)$,  with norm
$$\norm{u}_{X_1}=\left\lbrace  \int _a^m   ph^2  | \left (   \dfrac uh\right ) ' | ^2 + q_{1,h} | u|^2 \right\rbrace ^{1/2} =: t^R[u]^{1/2}$$
and 
$$X_1 =\{ u: u \in AC_{loc} (a,m): \norm{u}_{X_1} <\infty\},$$
with the norm $\|\cdot\|_{X_1}$.

Arlynskii expresses $T$ in \textit{divergence form}; this requires $T$ to be put in the form
\[
T= L_2^*QL_1,
\] 
where 
\begin{enumerate}
	\item
$L_1,~L_2$ are closed, densely defined operators with domains in $H:= L^2(a,m;k)$ and ranges in $ \mathcal{H}:= H\oplus H$;
\item $Q$ is a bounded and coercive operator on $H$;
\item $ \mathcal{D}(L_1) \cap  \mathcal{D}(L_2^*QL_2)$ is dense in $ \mathcal{D}(L_1)$.
\end{enumerate}  
In our application, 
\begin{align*}
&\mathcal{D}(L_1) = X_0;\ \ L_1 u = \left[\begin{array}{c}
u \\ hp^{1/2}\left( \frac{u}{h}\right) ' \end{array}\right ]; \\
&\mathcal{D}(L_2) = X_1;\ \ L_2 u = \left[\begin{array}{c}
u \\ hp^{1/2}\left( \frac{u}{h}\right) ' \end{array}\right ]; \\
& Q= \frac{1}{k}\left [   \begin{array}{ cc}   q_h & 0 \\ 0 & 1 \end{array}  \right ]
\end{align*}
where $q_h = q - \frac{(ph')'}{h}$. The adjoint operators are given by
\begin{align*}
& D(L^*_1)= H\oplus X_1, \; D(L^*_2)= H\oplus  X_0 ;\\
& L^*_j \frac{1}{k} \left [  \begin{array}{c}  f_1 \\f_2 \end{array}  \right ] = f_1-\frac{1}{h}\left(h p^{1/2}f_2\right)',\ \  \textrm{for}\ \  j=1,2 ,
\end{align*}
Then
\[
L_1^*Q^*L_2 u = \overline{q_h} u - \frac{1}{h}\left[ph^2 \left( \frac{u}{h}\right) '\right]'u = \tau^+ u = T^* u.
\]
Arlinskii's approach yields the following result; $\psi$ is the solution of $\tau^+~\psi = 0$ with $ \psi(a)=1$.

\begin{theorem}
	\label{Theorem 4.3} 
The formulae
\begin{align*} 
\mathcal{D}(\tilde T) ={}& \{     v\in X_1: {}&\\
v-(\psi-2 y)v(a) \in D(T_F); {}&
 \left[ p v' - p(\psi'-2y')v(a)-(\psi-2y)v(a)\right] (a)   =wv(a)\},\\
{}&  \tilde T v = \tau(v-(\psi-2 y)v(a)),
\end{align*}
establish a one to one correspondence between all coercive $m$-sectorial extensions $ \tilde{T}$ of $T$, excepting $T_F$ and $T_N$, and the set of all pairs $<w,y>$, where $w$ is a complex number with a positive real part, and $y\in X_0 $ satisfies
$$ \max \{ \Re\left[  t^R[\left( 2y-\varphi,\varphi\right) ],\varphi\in X_0\}\right]  < \Re w.$$
The associated closed form is given by 
$$\tilde t[v] = t^R[ v-(\psi-2 y)  v(a), v]+ wv(a) \overline { v(a)}, \ \  v\in X_1.$$

The Friedrichs and Krein-von Neumann extensions are determined by the pairs $<\infty,0>$ and $<0,0>$ respectively.

\end{theorem}

\subsection{The case $\dim( \ker T^*) =2$.}
 Let $\{\psi_1,\psi_2\}$ be a basis for $\ker~T^*$ and let $ \psi = \sum_{j=1}^2 c_j \psi_j$ be such that $\psi(a) = 1$. 
 This determines a vector in $\ker ~T^*$,
the complex number $w$ in the open right half plane determines a one-
dimensional coercive operator $W(\lambda\psi)=\lambda w \psi$ . In general, parameter $w$
can be  a  $2\times2$  sectorial/coercive matrix (as in  Theorem  \ref{Theorem 3.2}) which determines
a linear operator and in this case $y$ is a linear operator.

\bibliography{lms_extension}{}
\bibliographystyle{plain}



\end{document}